\documentclass[reqno,12pt]{amsart}

\usepackage[margin=1in]{geometry}
\usepackage{xspace}
\usepackage[utf8]{inputenc}
\usepackage{graphicx}
\usepackage{amssymb}
\usepackage{mathrsfs}
\usepackage[active]{srcltx}
\usepackage{url}
\usepackage{hyperref}
\usepackage{amsmath,amstext,amsxtra,amsgen,amsbsy,amsopn,amscd,amsthm,amsfonts}
\usepackage{latexsym}
\usepackage{dsfont}
\usepackage{enumitem}

\newtheorem{theorem}{Theorem}

\newtheorem{proposition}{Proposition}
\newtheorem{lemma}[proposition]{Lemma}

\theoremstyle{definition}

\theoremstyle{remark}

\newtheorem{remark}[proposition]{Remark}

\numberwithin{equation}{section}
\numberwithin{proposition}{section}

\newcommand\R{{\ensuremath {\mathbb R} }}

\newcommand\N{{\ensuremath {\mathbb N} }}
\newcommand\Z{{\ensuremath {\mathbb Z} }}

\renewcommand\phi{\varphi}

\renewcommand\le{\leqslant}
\renewcommand\ge{\geqslant}
\renewcommand\epsilon{\varepsilon}
\renewcommand\hat{\widehat}
\renewcommand\tilde{\widetilde}
\renewcommand\bar{\overline}

\newcommand{\cD}{\mathcal{D}}
\newcommand{\cH}{\mathcal{H}}
\newcommand\ii{{\ensuremath {\infty}}}

\newcommand{\cE}{\mathcal{E}}

\newcommand{\cC}{\mathcal{C}}

\newcommand{\cF}{\mathcal{F}}

\newcommand\1{{\ensuremath {\mathds 1} }}
\newcommand{\Sph}{\mathbb{S}}

\newcommand\cS{\mathcal{S}}

\DeclareMathOperator{\tr}{Tr}

\title{Compactness methods in Lieb's work}

\author{Julien Sabin}

\address{Julien Sabin, Centre de mathématiques Laurent Schwartz, UMR CNRS 7640, École polytechnique,91128 Palaiseau Cedex, France}
\email{julien.sabin@polytechnique.edu}

\begin{document}

\begin{abstract}
  We review some compactness methods appearing in the work of Lieb, with an emphasis on the techniques developed around his 1983 article on the optimizers for the Hardy-Littlewood-Sobolev inequality. 
\end{abstract}

\maketitle

\section{Introduction}

We survey several compactness methods appearing in Lieb's work. Such methods appear naturally when dealing with optimization problems: a natural way to prove the existence of optimizers is to show that optimizing sequences converge (perhaps up to a subsequence) by some compactness argument. The question of existence of optimizers appears in several works of Lieb: in chronological order, for the Young inequality \cite{BraLie-76}, in Choquard theory \cite{Lieb-77}, Hartree-Fock theory \cite{LieSim-77}, Thomas-Fermi theory \cite{LieSim-77b}, Thomas-Fermi-von Weizs\"acker theory \cite{BenBreLie-81}, for the Hardy-Littewood-Sobolev inequality \cite{Lieb-83b}, the Br\'ezis-Nirenberg problem \cite{BreNir-83}, for vector field equations \cite{BreLie-84}, in problems related to zero modes of magnetic fields \cite{FroLieLos-86}, for Gaussian kernel operators \cite{Lieb-90}, in non-relativistic QED \cite{GriLieLos-01,LieLos-03}, for Poincaré-type inequalities \cite{LieSeiYng-03}, for spherical Young's inequalities \cite{CarLieLos-04}, for the M\"uller functional \cite{FraLieSeiSie-07}, for the Hardy-Littlewood-Sobolev inequality on the Heisenberg group \cite{FraLie-12}, for the Pekar-Tomasevich functional at thereshold \cite{FraLieSei-12}, for negative ions at threshold \cite{BelFraLieSei-14}, for the liquid drop model \cite{FraLie-15}, for the Stein-Tomas inequality \cite{FraLieSab-16}, for Nash-type inequalities \cite{CarLie-19}. The previous list is probably not exhaustive, for instance we did not mention works related to spin systems where optimizers typically exist due the finite dimensionality of the problems. 

In this review, we focus on the techniques introduced by Lieb and his collaborators in a series of articles published around 1983, since they introduce new and important ideas that can be applicable to many problems. The key concept here is the loss (or lack) of compactness. As an illustration of this concept, let us briefly explain what were the compactness tools used by Lieb before 1983, to give a comparison point with the techniques we survey later on. The optimization problems we consider here are all of the form 
$$e_0:=\inf\Big\{ \cE(f)\ :\ f\in X_0\Big\}$$
where $X_0$ is a (non-empty) strongly closed subset of a (infinite-dimensional) Banach space $X$, and $\cE:X\to\R$ is a bounded-below, strongly continuous functional. To show that $e_0$ is attained (that is, there exists $f^*\in X_0$ with $\cE(f^*)=e_0$), it is natural to pick a minimizing sequence $(f_n)\subset X_0$ such that $\cE(f_n)\to e_0$ as $n\to\ii$, and the hope is that $f^*$ can be obtained as the limit of $(f_n)$ in some topology. The first obvious obstacle to this strategy is the fact that $X_0$ is in practice not compact for the strong topology of the (infinite dimensional space) $X$, so that a strongly converging subsequence of $(f_n)$ cannot be found in this way. The standard tool to bypass this difficulty is to use instead weak topologies on $X$ (either the weak-$*$ topology together with the Banach-Alaoglu theorem or the weak topology when $X$ is reflexive), for which sequential compactness is much easier to obtain. In all practical situations, one obtains in this way a subsequence of $(f_n)$ --that we still denote by $(f_n)$-- that converges weakly (or weakly-$*$) towards $f^*$ belonging to $\bar{X_0}^w$, the (sequentially) weak closure of $X_0$. To show that $f^*$ is an optimizer of the problem, it remains to show i) that $f^*\in X_0$ and ii) that $\cE(f^*)=\lim_{n\to\ii}\cE(f_n)$. The most favorable case where these properties hold is the case where $X_0$ is (sequentially) weakly closed, that is $\bar{X_0}^w=X_0$, and where $\cE$ is weakly lower semi-continuous (that is, for any sequence $(g_n)\subset X$ such that $g_n\rightharpoonup g$, one has $\liminf_{n\to\ii}\cE(g_n)\ge\cE(g)$); indeed, we then have $e_0=\lim_{n\to\ii}\cE(f_n)=\liminf_{n\to\ii}\cE(f_n)\ge\cE(f^*)\ge e_0$ since $f^*\in X_0$, so that $\cE(f^*)=e_0$ and $f^*$ is an optimizer for $e_0$. The first occurrence where loss of compactness occurs is the case where i) is not true; i.e. weak limits may leave the minimization set $X_0$. In this case, all hope is not necessarily lost since $(f_n)$ is not \emph{any} weakly converging sequence, it is a minimizing sequence for $e_0$. Hence, one may try to use this minimizing property to infer that $f^*\in X_0$. 

This is what happens for Thomas-Fermi \cite{LieSim-77b} or Hartree-Fock \cite{LieSim-77} theory. In the Thomas-Fermi problem (and in many other optimization problems), the minimization set $X_0$ is a sphere $X_0=\{f\in X\ :\ \|f\|_X=\lambda\}$ of radius $\lambda>0$ in the full space $X$, for which its weak closure is the ball of same radius $\bar{X_0}^w=\{f\in X\ :\ \|f\|_X\le\lambda\}$. In this kind of problems, showing that $f^*\in X_0$ means that a minimizing sequence cannot 'lose mass', i.e. that $\|f^*\|_X<\lambda$ is impossible. In the Hartree-Fock problem, the minimization set is not a sphere but rather the set of all rank-$N$ self-adjoint projections on $L^2(\R^3)$, for a fixed $N\in\N$. In this case, the weak closure is the set of all self-adjoint operators $\gamma$ with $0\le\gamma\le1$ and $\tr\gamma\le N$ and thus in this context, 'loss of mass' can be interpreted as a 'loss of particles' (since $\tr\gamma$ is interpreted as the number of particles that the state $\gamma$ describes). In both Thomas-Fermi and Hartree-Fock problems, the energy functional $\cE$ is weakly lower semi-continuous, so showing that $f^*\in X_0$ is enough to infer that $f^*$ is a minimizer. Actually, one can show that for $\lambda$ and $N$ not too big, 'loss of mass' does not happen and one indeed has $f^*\in X_0$. The proof depends quite strongly of the specific features of these problems, so that we will not provide it here. In the Thomas-Fermi problem for instance, it relies on the fact that for $\lambda$ not too big, the minimum $e_0$ is strictly decreasing as a function of the parameter $\lambda$, so that it is energetically unfavorable to loss mass. 

The situation is different for the Choquard problem \cite{Lieb-77}, where the minimization set is the ball $X_0=\{f\in H^1(\R^3)\ :\ \|f\|_{L^2}\le\lambda\}$ for some fixed $\lambda>0$, which is weakly-closed and hence minimizing sequence cannot leave the minimization set. What changes here is that the energy functional is \emph{not} weakly lower semi-continuous, due to the \emph{translation-invariance} of the problem. Indeed, for any $y\in\R^3$ and any $f\in X_0$, one can define the translate $\tau_y f:=f(\cdot-y)$ which clearly satisfies that $\tau_y f \in X_0$. Furthermore, the functional $\cE$ has the invariance $\cE(\tau_y f)=\cE(f)$ for any $f\in X_0$ and any $y\in\R^3$. Now notice that for any $(y_n)\subset\R^3$ such that $|y_n|\to\ii$ as $n\to\ii$, one has $f_n:=\tau_{y_n}f\rightharpoonup0$ weakly in $H^1(\R^3)$. Hence, if one chooses $f$ such that $\cE(f)<0$ (which in this case is always possible due to the explicit form of $\cE$), we have $\liminf_{n\to\ii}\cE(f_n)=\cE(f)<0=\cE(0)$ with $f_n\rightharpoonup0$, proving that $\cE$ is not weakly lower semi-continuous (so that point ii) above fails). Another point of view on this loss of compactness is that it is hopeless to try to obtain a minimizer as a weak limit of any minimizing sequence; because if it worked and one obtained a minimizer $f^*\in X_0$ in this way, then $(\tau_{y_n}f^*)_n$ would be a minimizing sequence whose weak limit (which is zero) is obviously not a minimizer. If one wants to obtain a minimizer as a weak limit of minimizing sequence, one thus cannot consider any minimizing sequence; one should look at a subclass of minimizing sequences for which translation-invariance is 'broken' and in such a way that lower semi-continuity is restored for this type of sequences. Lieb did so by showing that one can look for minimizers among the subclass of functions in $X_0$ which are \emph{radially symmetric decreasing}, for which he proved that the energy was lower semi-continuous and hence its weak limits are minimizers for the Choquard problem.

The starting point of this review is to consider minimization problems where both the issues presented above are present: the minimization set is not weakly closed and the problem has some 'non-compact' invariance. In the context of Lieb's work, the first such problem where this kind of phenomenon appeared was the existence of optimizers for the Hardy-Littlewood-Sobolev inequality, 
$$\sup\Big\{ \|f*|\cdot|^{-\lambda}\|_{L^q}\ :\ f\in L^p(\R^d),\ \|f\|_{L^p}=1\Big\}$$
with $d\ge1$, $\lambda\in(0,d)$, $1<p,q<+\ii$ such that $1/p+\lambda/d=1+1/q$. The maximization set is here a sphere (which is again not weakly-closed), and the energy functional is invariant under translations and dilations.  In the seminal work \cite{Lieb-83b}, Lieb understood how to consider minimizing sequences such that these invariances have been broken, and that do not lose mass and hence converge towards maximizers. This approach was applied in the next years to several other optimization problems \cite{BreNir-83,BreLie-84,FroLieLos-86}. The goal of this review is to explain these techniques and how they can be apply to several optimization problems. We emphasize that our focus is on the works of Lieb on these questions and that many other authors contributed to this theory, before and after the articles that we survey. For instance, Lions developed in the same period his technique (which shares several common features with Lieb's approach) of concentration-compactness \cite{Lions-84,Lions-84b,Lions-85,Lions-85b}, bringing even more tools to tackle such problems. To keep this review at a reasonable size, we will not try to comment on these works and to compare them to the results we detail here. 

In Section \ref{sec:tools}, we present the specific tools developed by Lieb and his collaborators on the question of loss of compactness. In Section \ref{sec:app-exact-sym}, we explain how these tools were applied to solve various optimization problems, including some other well-known optimization problems that we can treat in the same way. In Section \ref{sec:app-approx-symp}, we apply the same methods but to slightly different problems with a non-compact 'almost' invariance. In Appendix \ref{app:profile}, we explain how these techniques can be extended to understand profile decompositions.

\medskip

\noindent\textbf{Acknowledgements:} The author is indebted to Mathieu Lewin and Rupert Frank for introducing him to these techniques, for many discussions about them, and for comments on this text. He is of course also very thankful to Elliott Lieb for all the inspiring ideas and works that are presented here. 

\section{Optimization toolbox}\label{sec:tools}

We choose to start by highlighting the various tools that Lieb and his collaborators developed in the context of compactness problems. More precisely, we present an abstract result on convergence of optimizing sequences for functional inequalities \cite[Lemma 2.7]{Lieb-83b}, which itself relies on a refinement of Fatou's lemma called the Br\'ezis-Lieb lemma \cite{BreLie-83}. We also present a very useful result about the existence of non-zero weak limits in Sobolev spaces \cite[Lemma 6]{Lieb-83}, which is often combined with the $pqr$ lemma \cite[Lemma 2.1]{FroLieLos-86} (see also \cite[Lemma 2.1]{BreLie-84}). 

As we already mentioned, the Hardy-Littewood-Sobolev was chronologically the first optimization problem to which these methods were applied. We will see that many functional inequalities can be treated with the same strategies. The following result provides a fairly general setting in which some maximizing sequences converge to maximizers in the context of sharp constants for functional inequalities.

\begin{proposition}\label{prop:MMM-pq}{\cite[Lemma 2.7]{Lieb-83b}}
 Let $0<p<q$ and $X$, $Y$ be measure spaces. Assume that $A$ is a non-zero bounded linear operator from $L^p(X)$ to $L^q(Y)$ and denote 
 $$\cS:=\sup\{\|Af\|_{L^q(Y)}^q\ |\ f\in L^p(X),\ \|f\|_{L^p(X)}=1\}\in(0,+\ii).$$
 Assume that there exist $(f_n)\subset L^p(X)$ and $f\in L^p(X)$ such that 
 \begin{enumerate}
  \item $\|f_n\|_{L^p}=1$ for all $n$;
  \item $\|Af_n\|_{L^q}^q\to\cS$;
  \item $f_n\to f$ almost everywhere;
  \item $f\neq0$;
  \item $Af_n\to Af$ almost everywhere.
 \end{enumerate}
 Then, $(f_n)$ converges strongly in $L^p(X)$ to $f$ and $\|Af\|_{L^q}^q=\cS$.
\end{proposition}

The interpretation of Proposition \ref{prop:MMM-pq} is that if one can find a maximizing sequence for $\cS$ for which conditions (3), (4), (5) are satisfied, then it automatically converges towards a maximizer for $\cS$. Let is comment on these conditions. Condition (4) is natural from the considerations explained in the introduction, in the sense that in the context of non-compact invariances, it is easy to construct potential maximizing sequences which converge to zero. Hence, the assumption (4) is a manifestation that such potential invariances are 'broken'. Here, notice that the minimization set is a sphere so that its weak closure is the closed ball of same radius. As we explained in the introduction, it is natural to try to prove that minimizing sequences cannot leave the minimization set by showing that their weak limits cannot belong to the open ball of same radius. In this respect, condition (4) may look incomplete because while it forbids the limit to be zero, it does not a priori forbid that it belongs to an intermediate sphere of radius strictly between $0$ and $1$. Such a scenario may again happen in the presence of non-compact invariances, where a potential maximizing sequence is composed of a positive piece of mass which remains in the limit and of another positive piece of mass escaping along the non-compact invariances. To show that maximizing sequences cannot exhibit this mass-splitting phenomenon, conditions (3) and (5) intervene: we will see below that, together with the Br\'ezis-Lieb lemma, they exactly show that mass splitting is energetically unfavorable. 

\begin{remark}
The assumption $p<q$ in Proposition \ref{prop:MMM-pq} is natural when for instance $A$ is translation-invariant \cite[Theorem 1.1]{Hormander-60}.
\end{remark}

Before proving Proposition \ref{prop:MMM-pq}, let us state the important result on which it relies. 

\begin{lemma}[Br\'ezis-Lieb lemma \cite{BreLie-83,Lieb-83b}]\label{lem:bre-lie}
 Let $r>0$ and $Z$ be a measure space. Assume that $(g_n)\subset L^r(Z)$ is bounded and converges almost everywhere to $g\in L^r(Z)$. Then, we have 
 $$\int_Z |g_n|^r = \int_Z|g|^r + \int_Z|g_n-g|^r +o_{n\to+\ii}(1).$$ 
\end{lemma}

The Br\'ezis-Lieb lemma is a refinement of Fatou's lemma, since it quantifies how non-negative is the quantity $\liminf_{n\to\ii}\int_Z |g_n|^r-\int_Z|g|^r$. Notice that the assumption of the Br\'ezis-Lieb lemma is almost everywhere convergence, which motivates the assumptions (3) and (5) in Proposition \ref{prop:MMM-pq}. Notice also that a Br\'ezis-Lieb lemma where almost everywhere convergence is replaced by weak convergence holds only for $r=2$ (see \cite[Eq. (4.23)]{KilVis-book}), showing the importance of a.e. convergence for this kind of problems.

We now explain why the Br\'ezis-Lieb lemma implies Proposition \ref{prop:MMM-pq}. We provide the proof because its idea can be adapted to various situations besides the one presented here, and because the role of each assumption is made more explicit. 

\begin{proof}[Proof of Proposition \ref{prop:MMM-pq}]
Applying Lemma \ref{lem:bre-lie} to $(f_n)$ implies that 
$$\lim_{n\to+\ii}\|f_n-f\|_{L^p}^p = 1-\|f\|_{L^p}^p.$$
Now applying again Lemma \ref{lem:bre-lie} but now to $(Af_n)$ we deduce that 
\begin{align*}
 \cS + o(1) &= \|Af_n\|_{L^q}^q \\
 &= \|Af\|_{L^q}^q + \|A(f_n-f)\|_{L^q}^q+o(1) \\
 &\le \cS\big( \|f\|_{L^p}^q + \|f_n-f\|_{L^p}^q \big) + o(1).
\end{align*}
In the limit $n\to+\ii$ and using $\cS>0$ we obtain 
$$1\le \|f\|_{L^p}^q + \big(1-\|f\|_{L^p}^p\big)^{q/p}.$$
Since $0<p<q$, we have that for any $a,b\ge0$, $(a+b)^{q/p}\le a+b$ with equality if and only if $a=0$ or $b=0$. Applying this fact to $a=\|f\|_{L^p}^p$ and $b=1-\|f\|_{L^p}^p$ (which is non-negative by Fatou's lemma), we obtain
$$1\le \|f\|_{L^p}^q + \big(1-\|f\|_{L^p}^p\big)^{q/p}\le 1,$$
so that we are in the equality case $\|f\|_{L^p}^p=0$ or $1-\|f\|_{L^p}^p=0$. Since $f\neq0$, we deduce that $\|f\|_{L^p}=1$, and hence $1-\|f\|_{L^p}^p=\lim_{n\to+\ii}\|f_n-f\|_{L^p}^p=0$, showing that $f_n\to f$ strongly. Using the continuity of $A$, this implies that $Af_n\to Af$ strongly in $L^q$, and hence $\cS=\lim_{n\to+\ii}\|Af_n\|_{L^q}^q=\|Af\|_{L^q}^q$.
\end{proof}

\begin{remark}
 In \cite{Lieb-83b}, the last argument proving that $f_n\to f$ in $L^p(X)$ is attributed to Br\'ezis and in \cite{BreNir-83} the same argument adapted to the Hilbert space setting described below is attibuted to Browder.
\end{remark}

\begin{remark}
 The above proof can be adapted to the case $q=p$ to deduce that $f$ is a maximizer for $\cS$, without knowing that $f_n\to f$ strongly in $L^p$. This is done using only the bound $\|A(f_n-f)\|_{L^q}^q\le\cS\|f_n-f\|_{L^q}^q$, so that 
 $$\cS+o(1)\le \|Af\|_{L^q}^q+\cS(1-\|f\|_{L^q}^q)+o(1)$$
 and hence $\|Af\|_{L^q}^q\ge\cS\|f\|_{L^q}^q$.
\end{remark}

In Proposition \ref{prop:MMM-pq}, the assumptions (1) and (2) are automatic for maximizing sequences for $\cS$. Assumptions (3) to (5) are less obvious to obtain in practice. In particular, assumption (3) does not follow from the boundedness of $(f_n)$ in $L^p$ and one needs specific properties of $A$ to obtain it. In the case $p=2$, one can get rid of assumption (3) by the following straightforward adaptation of Proposition \ref{prop:MMM-pq}, which appears explicitly in \cite[Prop. 1.1]{FanVegVis-11}.

\begin{proposition}\label{prop:MMM-2q}
 Let $q>2$ and $X$, $Y$ be measure spaces. Assume that $A$ is a non-zero bounded linear operator from $L^2(X)$ to $L^q(Y)$ and denote 
 $$\cS:=\sup\{\|Af\|_{L^q(Y)}^q\ |\ f\in L^2(X),\ \|f\|_{L^2(X)}=1\}\in(0,+\ii).$$
 Assume that there exist $(f_n)\subset L^2(X)$ and $f\in L^2(X)$ such that 
 \begin{enumerate}
  \item $\|f_n\|_{L^2}=1$ for all $n$;
  \item $\|Af_n\|_{L^q}^q\to\cS$;
  \item $f_n\rightharpoonup f$ in $L^2(X)$;
  \item $f\neq0$;
  \item $Af_n\to Af$ almost everywhere.
 \end{enumerate}
 Then, $(f_n)$ converges strongly in $L^2(X)$ to $f$ and $\|Af\|_{L^q}^q=\cS$.
\end{proposition}

One can see the advantage of Proposition \ref{prop:MMM-2q} compared to \ref{prop:MMM-pq}: the assumption (3) is now a consequence of the boundedness of $(f_n)$ in $L^2$, using the weak compactness of closed balls in $L^2$. In the context of Proposition \ref{prop:MMM-2q}, only assumptions (4) and (5) have to be checked in practice.

\begin{proof}[Proof of Proposition \ref{prop:MMM-2q}]
 The proof is identical to the one of Proposition \ref{prop:MMM-pq}, replacing the application of the Br\'ezis-Lieb lemma to $(f_n)$ by the following consequence of the weak convergence of $(f_n)$ to $f$:
 $$\lim_{n\to+\ii}\|f_n-f\|_{L^2}^2 = 1-\|f\|_{L^2}^2.$$
\end{proof}

\begin{remark}
 As is clear from the proof, Proposition \ref{prop:MMM-2q} can be extended straightforwardly to the case where $L^2(X)$ is replaced by any separable Hilbert space $\cH$, as in \cite{FanVegVis-11}.
\end{remark}

The next tool we present is related to verifying the key assumption (4) in Proposition \ref{prop:MMM-2q}, namely the existence of maximizing sequences which have a non-zero weak limit. As we already mentioned, this fact is particularly relevant in problems with a non-compact invariance where optimizing sequences with zero weak limit are bound to exist (just by applying the symmetries to an optimizer). The following result, which appears for the first time in \cite[Lemma 6]{Lieb-83}, is a very useful tool to obtain non-zero weak limits for bounded sequences in $W^{1,p}(\R^d)$. 

\begin{proposition}\label{prop:non-zero-weak-W1p}
 Let $d\ge1$ and $1<p<+\ii$. Assume that $(f_n)\subset W^{1,p}(\R^d)$ is a bounded sequence such that there exist $\delta,\epsilon>0$ such that for all $n$, $|\{|f_n|>\epsilon\}|\ge\delta$. Then, there exists $(x_n)\subset\R^d$ such that $(f_n(\cdot-x_n))$ has a non-zero weak limit in $W^{1,p}(\R^d)$.
\end{proposition}

\begin{remark}
 In the above result, the Sobolev space $W^{1,p}(\R^d)$ may be replaced by its homogeneous version $\dot{W}^{1,p}(\R^d)$.
\end{remark}

We will present below some applications of Proposition \ref{prop:non-zero-weak-W1p} in the case $p=2$ to obtain Assumption (4) in Proposition \ref{prop:MMM-2q} in its Hilbert space version with $\cH=H^1(\R^d)$. Notice that Proposition \ref{prop:non-zero-weak-W1p} naturally breaks translation invariance to obtain a non-zero weak limit. In \cite{Lieb-83}, one proof of Proposition \ref{prop:non-zero-weak-W1p} is done using some comparison between the lowest Dirichlet eigenvalues of two domains, and another, due to Br\'ezis, is done via localization on cubes and the Sobolev inequality.

The main assumption on the level sets of $(f_n)$ in Proposition \ref{prop:non-zero-weak-W1p} can be obtained in practice via the following $pqr$ lemma \cite[Lemma 2.1]{FroLieLos-86} (see also \cite[Lemma 2.1]{BreLie-84}).

\begin{lemma}[$pqr$ lemma]
 Let $X$ be a measure space and let $0<p<q<r<+\ii$. Let $(f_n)$ be a bounded sequence in $L^p(X)\cap L^r(X)$ such that there exists $\alpha>0$ such that for all $n$, $\|f_n\|_{L^q}\ge\alpha$. Then, there exist $\delta,\epsilon>0$ such that for all $n$, $|\{|f_n|>\epsilon\}|\ge\delta$.
\end{lemma}

\begin{proof}
 Let $\epsilon>0$. For all $n$, we have 
 \begin{align*}
    \alpha^q &\le \int_X |f_n|^q = \int_{|f_n|\le\epsilon}|f_n|^q + \int_{\epsilon<|f_n|<1/\epsilon}|f_n|^q+\int_{|f_n|\ge1/\epsilon}|f_n|^q \\
    &\le \epsilon^{q-p}\|f_n\|_{L^p}^p+\epsilon^{-q}|\{|f_n|>\epsilon\}|+\epsilon^{r-q}\|f_n\|_{L^r}^r \\
    &\le C(\epsilon^{q-p}+\epsilon^{r-q})+\epsilon^{-q}|\{|f_n|>\epsilon\}|,
 \end{align*}
 leading to the result for $\epsilon$ small enough.
\end{proof}

Combining Proposition \ref{prop:non-zero-weak-W1p} and the $pqr$ lemma leads to the following useful result. 

\begin{proposition}\label{prop:dicho-W1p}
 Let $d\ge1$ and $1<p<+\ii$. Let $q\in(p,+\ii)$ such that $1/q>1/p-1/d$. Assume that $(f_n)$ is a bounded sequence in $W^{1,p}(\R^d)$. Then, we have the following alternative:
 \begin{enumerate}
  \item either $f_n\to0$ in $L^q(\R^d)$; 
  \item or there exists $(x_n)\subset\R^d$ such that, up to a subsequence, $(f_n(\cdot-x_n))$ has a non-zero weak limit in $W^{1,p}(\R^d)$.
 \end{enumerate}
\end{proposition}

\begin{proof}[Proof of Proposition \ref{prop:dicho-W1p}]
  By Sobolev's embedding, there exists $q^*>q$ such that $(f_n)$ is bounded in $L^p(\R^d)\cap L^{q^*}(\R^d)$. If $(f_n)$ does not converge to zero in $L^q(\R^d)$, then there exists $\alpha>0$ such that, up to a subsequence, we have $\|f_n\|_{L^q}\ge\alpha$ for all $n$. By the $pqr$ lemma, $(f_n)$ thus satisfies the assumptions of Proposition \ref{prop:non-zero-weak-W1p} and hence admits non-zero weak limits up to translations.
\end{proof}

For optimization problems posed on $H^1(\R^d)$, Proposition \ref{prop:MMM-2q} and Proposition \ref{prop:dicho-W1p} provide a roadmap to obtain the existence of optimizers: Assumption (4) would follow if optimizing sequences do not converge to zero in some $L^q(\R^d)$ for a subcritical $q$, and there remains to understand why Assumption (5) holds. We now explain how to apply these tools in some explicit cases.

\section{Applications to problems with exact symmetries}\label{sec:app-exact-sym}

We apply the compactness tools introduced in the previous section to solve several optimization problems. We begin with the historical example of Lieb \cite{Lieb-83b} concerning the Hardy-Littlewood-Sobolev inequality, emphasizing how the invariances are broken in conjonction with Proposition \ref{prop:MMM-pq} to obtain existence of optimizers. We then shift to the Gagliardo-Nirenberg-Sobolev inequality to illustrate how to apply to $H^1$-methods above, as an introduction to the work of Br\'ezis-Lieb on vector field equations \cite{BreLie-84}. Finally, we explain how to apply similar ideas to the Sobolev, (generalized) Gagliardo-Nirenberg-Sobolev, and Strichartz inequalities. 

\subsection{Hardy-Littlewood-Sobolev inequality}

Let $d\in\N^*$. Let $1<p,q<+\ii$ and $0<\lambda<d$ be such that $1/p + \lambda/d = 1+1/q$. Consider the optimization problem
$$\cS:=\sup\Big\{ \|f*|\cdot|^{-\lambda}\|_{L^q(\R^d)}^q\ :\ f\in L^p(\R^d),\ \|f\|_{L^p(\R^d)}=1\Big\},$$
corresponding to the sharp constant in the Hardy-Littlewood Sobolev inequality. As we already mentioned, this problem exhibits both translation and dilation invariance (meaning that both the optimization set and the energy functional are invariant under the transformations $f(x)\to f(x+x_0)$ and $f(x)\to \delta^{d/p}f(\delta x)$ for any $x_0\in\R^d$ and any $\delta>0$). Despite this difficulty, Lieb proved existence of maximizers:

\begin{theorem}{\cite[Theorem 2.3]{Lieb-83b}}\label{thm:HLS-pq}
  There exists $f\in L^p(\R^n)$ with $\|f\|_{L^p(\R^d)}=1$ such that $\|f*|\cdot|^{-\lambda}\|_{L^q(\R^d)}^q=\cS$.
\end{theorem}

The proof of Theorem \ref{thm:HLS-pq} consists in applying Proposition \ref{prop:MMM-pq}. The way to obtain Assumption (3) here is to resort to rearrangement inequalities \cite[Lemma 2.1]{Lieb-83b} to infer that there exists a maximizing sequence $(f_n)\subset L^p(\R^d)$ for $\cS$ such that for any $n$, $f_n$ is radially decreasing and non-negative (indeed, for any $f\in L^p(\R^d)$, one has $\|f*|\cdot|^{-\lambda}\|_{L^q(\R^d)}\le \|(f^*)*|\cdot|^{-\lambda}\|_{L^q(\R^d)}$ and $\|f\|_{L^p}=\|f^*\|_{L^p}$ where $f^*$ is the symmetric decreasing rearrangement of $f$ \cite[Definition 1]{Lieb-83b}). Notice that this remark breaks the translation-invariance of the problem, and to break the dilation invariance we use the following argument. Abusing notations, we still denote by $f_n$ the function of the radius so that $f_n:x\mapsto f_n(|x|)$. The next result enables to deal with dilations and to obtain Assumption (4), i.e. to prove that $f_n$ has a non-zero pointwise limit. 

\begin{proposition}{\cite[Lemma 2.4]{Lieb-83b}}\label{prop:refined-HLS-pq}
 There exists $C>0$ and $\theta\in(0,1)$ such that for any non-negative and radially decreasing $f\in L^p(\R^d)$ we have 
 $$\| f * |\cdot|^{-\lambda} \|_{L^q} \le C\Big(\sup_{r>0}r^{d/p}f(r)\Big)^\theta\|f\|_{L^p}^{1-\theta}.$$
\end{proposition}

Such an inequality may be called a refined HLS inequality because it implies the actual HLS inequality, since $\sup_{r>0}r^{d/p}f(r)\le c\|f\|_{L^p}$ due to the fact that $f$ is radially decreasing. Applying Proposition \ref{prop:refined-HLS-pq} to $(f_n)$ and using that $\cS>0$, we infer that there exist $c>0$ such that for all $n$ one has $\sup_{r>0}r^{d/p}f_n(r)>c$. Hence, for any $n$ there exists $r_n>0$ such that $r_n^{d/p}f_n(r_n)\ge c/2$. Replacing $f_n$ by $r_n^{d/p}f_n(r_n\cdot)$ --and still denoting this sequence by $(f_n)$--, we obtain a sequence of non-negative radially decreasing functions $(f_n)$ which is maximizing for $\cS$ and such that $f_n(1)\ge c/2$ for all $n$. Since $\|f_n\|_{L^p}=1$, we have $f_n(r)\le(|\Sph^{d-1}|/d)^{-1/p}r^{-n/p}$ for all $n$ and $r$, so that using Helly's selection principle there exists a non-negative radially decreasing function $f\in L^p(\R^d)$ with $\|f\|_{L^p}\le 1$ such that (up to a subsequence) $f_n(r)\to f(r)$ as $n\to+\ii$ for all $r>0$. In particular, we have $f(1)\ge c/2>0$ implying that $f\neq0$ and thus recovering Assumption (4) of Proposition \ref{prop:MMM-pq}. There remains to prove Assumption (5) of Proposition \ref{prop:MMM-pq}, which follows from the elementary

\begin{lemma}
 Let $(f_n)\subset L^p(\R^d)$ be a sequence of non-negative radially decreasing functions such that $\|f_n\|_{L^p}=1$ for all $n$ and such that $f_n\to f$ almost everywhere. Then, $f_n*|\cdot|^{-\lambda}\to f*|\cdot|^{-\lambda}$ almost everywhere. 
\end{lemma}

\begin{proof}
 Recall that $f_n(r)\le cr^{-d/p}$ with $c$ independent of $n$. Then, the result follows from dominated convergence, noticing that for any $x\neq0$, $y\mapsto|y|^{-d/p}|x-y|^{-\lambda}$ is integrable on $\R^d$. 
\end{proof}

It remains to give the 

\begin{proof}[Proof of Proposition \ref{prop:refined-HLS-pq}]
 Denoting by $A$ the operator $f\mapsto f*|\cdot|^{-\lambda}$ and by $U_s:L^s_{\rm rad}(\R^d)\to L^s(\R)$ the isometric isomorphism $(U_s f)(u)=e^{du/s}f(e^u)$, we have $U_q A U_p^*: g\mapsto g*L_d$ with 
 $$\forall u\in\R,\ L_d(u):=e^{du/q}\int_{\Sph^{d-1}}\frac{d\omega}{|e^{u}e_1-\omega|^\lambda}=c_de^{du/q}\int_0^\pi\frac{(\sin\theta)^{d-2}\,d\theta}{((e^u-\cos\theta)^2+\sin^2\theta)^{\lambda/2}}.$$
 Since 
 $$L_d(u)\sim
 \begin{cases}
  c e^{du/q} & \text{as}\ u\to-\ii,\\
  c \begin{cases} 1 & \text{if}\ \lambda<d-1 \\ \log(1/u) & \text{if}\ \lambda=d-1\\ u^{d-1-\lambda} & \text{if}\ \lambda>d-1 \end{cases} & \text{as}\ u\to0,\\
  c e^{(d/q-\lambda)u} &\text{as}\ u\to+\ii,
 \end{cases}
 $$
 the function $L_d$ is integrable on $\R$. Hence, $U_q A U_p^*$ is bounded on $L^s(\R^d)$ for all $s\in[1,+\ii]$ and using that $q\in(p,+\ii)$, we deduce by H\"older's inequality that there exist $C>0$ and $\theta\in(0,1)$ such that 
 $$\forall g\in L^p(\R)\cap L^\ii(\R),\ \| U_q A U_p^*g\|_{L^q} \le \| U_q A U_p^*g\|_{L^p}^{1-\theta}\| U_q A U_p^*g\|_{L^\ii}^\theta\le C\|g\|_{L^p}^{1-\theta}\|g\|_{L^\ii}^\theta.$$
 Stating this inequality in terms of the operator $A$, we get the result.
\end{proof}

\begin{remark}
 A consequence of the previous proof and of Proposition \ref{prop:MMM-pq} is that any maximizing sequence $(f_n)$ for $\cS$ admits a subsequence such that there exists $(\delta_n)\subset(0,+\ii)$ such that $(\delta_n^{d/p}f_n^*(\delta_n\cdot))$ converges strongly in $L^p(\R^d)$ (and hence the limit is a maximizer for $\cS$). Lions \cite[Theorem 2.1]{Lions-85b} proved that any maximizing sequence $(f_n)$ for $\cS$ admits a subsequence such that there exists $(x_n)\subset\R^d$ and $(\delta_n)\subset(0,+\ii)$ such that $(\delta_n^{d/p}f_n(\delta_n(\cdot-x_n)))$ converges strongly in $L^p(\R^d)$ (that is, we use translations instead of rearrangement).
\end{remark}

\subsection{Gagliardo-Nirenberg-Sobolev inequality}

In the above approach to the HLS inequality, rearrangement was used to break the translation-invariance of the problem. In the following, we will see that the same method can be applied to prove the existence of optimizers for the model problem of Gagliardo-Nirenberg-Sobolev inequality. In one space dimension, this was done explicitly by Lieb in \cite[Theorem 4.2]{Lieb-83b}, but his method (which was used before in \cite[Theorem 7]{Lieb-77} and \cite[Appendix A]{LieOxf-81}) works in any dimension. Another proof using rearrangement was suggested in \cite[Remark (A), Sec. 4]{LieThi-76}, using rather the methods of \cite{GlaMarGroThi-76,Talenti-76} relying on solving the Euler-Lagrange equations among radial decreasing functions. We also provide another method based on the $H^1$-based tools of Section \ref{sec:tools}, and compare the two approaches. 

Let $d\ge1$ and $q\in(2,+\ii)$ be such that $1/q>1/2-1/d$. Define 
$$\cS:=\sup\left\{ \int_{\R^d}|u|^q\ :\ u\in H^1(\R^d),\ \|u\|_{H^1}=1\right\}.$$

\begin{proposition}\label{prop:sharp-GNS}
 There exists $u\in H^1(\R^d)$ with $\|u\|_{H^1}=1$ such that $\|u\|_{L^q}^q=\cS$.
\end{proposition}

\begin{proof}
    Let $(u_n)\subset H^1(\R^d)$ be a maximizing sequence for $\cS$. As was noticed by Lieb \cite[Lemma 4.1]{Lieb-83b} (see also \cite[Lemma 5]{Lieb-77} or \cite[Theorem 1]{GlaMarGroThi-76} for earlier proofs), since $\|\nabla f\|_{L^2}\ge \|\nabla f^*\|_{L^2}$ for any $f\in H^1(\R^d)$, one may assume that $u_n$ is radially symmetric decreasing for all $n$. Since $(u_n)$ is bounded in $H^1(\R^d)$, it converges weakly to some $u\in H^1(\R^d)$ up to a subsequence, and by the Rellich-Kondrachov theorem one may also assume that $u_n\to u$ a.e. In particular, $\|u\|_{H^1}\le1$. From the boundedness of $(u_n)$ in $H^1(\R^d)$, one also deduces that $(u_n)$ is bounded in $L^p(\R^d)$ for any $p\ge2$ such that $1/p> 1/2-1/d$. Applying this fact for $p=2$ and $p=q_0$ for some fixed $q_0>q$ with $1/q_0>1/2-1/d$, one deduces from the fact that $(u_n)$ is radially symmetric decreasing that (as in the HLS case) for all $n$ and all $r>0$, $|u_n(r)|\le C\min(r^{-d/2},r^{-d/(2q_0)})\in L^q(\R^d)$. Hence, by the dominated convergence theorem, we deduce that $u_n\to u$ in $L^q(\R^d)$ and hence $\cS=\|u\|_{L^q}^q\le \cS\|u\|_{H^1}^q$ so that $\|u\|_{H^1}\ge1$ and hence $\|u\|_{H^1}=1$. Finally, $u$ is a maximizer for $\cS$.
\end{proof}

As in the HLS inequality, the above method relies on rearrangement to break translation invariance and actually shows that for any maximizing sequence $(u_n)$, there exists a (radially symmetric decreasing) maximizer $u$ for $\cS$ such that $(u_n^*)$ converges strongly to $u$ in $H^1(\R^d)$, up to a subsequence. We now present an alternative approach to Proposition \ref{prop:sharp-GNS}, in the spirit of \cite{BreLie-84} (see also \cite[Lemma 4.2]{CarFraLie-14}), that does not rely on rearrangement but rather on the tools of Section \ref{sec:tools} and has the advantage of describing more precisely maximizing sequences. It can also be applied to problems where rearrangement is not available, as we present in the next example. 

\begin{proof}[Alternative proof of Proposition \ref{prop:sharp-GNS}]
  Let $(u_n)$ be a maximizing sequence for $\cS$. In particular, $\|u_n\|_{L^q}^q\to\cS$ as $n\to\ii$ and thus $u_n$ does not converge to zero in $L^q(\R^d)$. By Proposition \ref{prop:dicho-W1p} applied to $p=2$, there exist $(x_n)\subset\R^d$ and $u\in H^1(\R^d)\setminus\{0\}$ such that $v_n:=u_n(\cdot-x_n)$ converges weakly to $u$ in $H^1(\R^d)$. The sequence $(v_n)$ is also a maximizing sequence for $\cS$, and by the Rellich-Kondrachov theorem we can also assume that $v_n\to u$ almost everywhere. We can then apply Proposition \ref{prop:MMM-2q} in its Hilbert space version with $\cH=H^1(\R^d)$ to the sequence $(v_n)$ and the operator $A:f\in H^1(\R^d)\mapsto f\in L^q(\R^d)$ to deduce that $u$ is a maximizer for $\cS$ and that $(v_n)$ converges strongly in $H^1(\R^d)$. 
\end{proof}

A corollary of the above proof is that for any maximizing sequence $(u_n)$ for $\cS$, there exist $(x_n)\subset\R^d$ and a maximizer $u$ of $\cS$ such that $(u_n(\cdot-x_n))$ converges strongly to $u$ in $H^1(\R^d)$, up to a subsequence. In other words, maximizing sequences converge strongly up to translations. This powerful statement (in conjunction with the further determination of the set of maximizers) can be used in the context of the associated nonlinear Schr\"odinger equation, to obtain the stability of standing waves as showed by Cazenave and Lions \cite{CazLio-82} or to study the minimal mass blow-up solutions following Weinstein \cite{Weinstein-86} (see also Merle \cite{Merle-93}).

\subsection{Solutions to vector field equations}

We present now a problem where rearrangement techniques do not apply and one must use the second approach of the last subsection. In the article \cite{BreLie-84}, Br\'ezis and Lieb proved the existence of optimizers for 'vector field' problems of the type 
$$\cS=\sup\left\{ \int_{\R^d}G(u(x))\,dx\ :\ u\in\cC,\ \|\nabla u\|_{L^2}=1\right\},$$
where $\cC=\{u\in\dot{H}^1(\R^d,\R^N)\ :\ G(u)\in L^1(\R^d)\}$, $d\ge3$, $N\ge1$, and $G:\R^N\to\R$ is a non-zero continuous function satisfying some assumptions that we detail below. For general $G$, this problem has a more general structure than the ones covered by Proposition \ref{prop:MMM-2q}, but we will see that some ideas still apply here. Notice that the problem $\cS$ is invariant under translations, and that for $N\ge2$ the rearrangement techniques used above cannot be used. To break translation invariance, Br\'ezis and Lieb rathered appealed to Proposition \ref{prop:non-zero-weak-W1p}. To illustrate how to adapt the tools of Section \ref{sec:tools} in this more general setting, assume that $G$ satisfies the assumptions 
\begin{enumerate}
 \item $\forall u\in\R^N\setminus\{0\},\ G(u)>0$ and $G(0)=0$,
 \item $\lim_{|u|\to0}|u|^{-p} G(u)=0=\lim_{|u|\to\ii}|u|^{-p}G(u),\quad p=\frac{2d}{d-2}$,
 \item $\forall\delta>0,\ \exists C_\delta>0,\ \forall u,v\in\R^N,\ |G(u+v)-G(u)|\le\delta|u|^p+C_\delta|v|^p$.
\end{enumerate}
Notice that assumptions (1) and (2) together with Sobolev's embedding imply that $\cC=\dot{H}^1(\R^d,\R^N)$.

\begin{proposition}
 Under the above assumptions on $G$, there exists $u\in\cC$ with $\|\nabla u\|_{L^2}=1$ such that $\int G(u) = \cS$. 
\end{proposition}

\begin{proof}
 Let $(u_n)$ be a maximizing sequence for $\cS$. Let $\eta>0$ and from assumption (2) on $G$, let $\epsilon>0$ such that $G(u)\le\eta|u|^p$ for all $|u|\le\epsilon$ or $|u|\ge1/\epsilon$. Hence, for $n$ large enough we have by Sobolev's embedding 
 \begin{align*}
  \cS/2 \le \int G(u_n) &= \int_{|u_n|\le\epsilon}G(u_n) + \int_{\epsilon<|u_n|<1/\epsilon}G(u_n)+\int_{|u_n|\ge1/\epsilon} G(u_n) \\
  &\le C\eta+\left(\max_{\epsilon\le|u|\le1/\epsilon}G(u)\right)|\{|u_n|>\epsilon\}|
 \end{align*}
 so that for $\eta$ small enough we have $|\{|u_n|>\epsilon\}|\ge\alpha$ for all $n$ for some $\alpha>0$ independent of $n$. Using Proposition \ref{prop:non-zero-weak-W1p}, we obtain $(x_n)\subset\R^d$ such that $v_n:=u_n(\cdot-x_n)$ converges weakly to some $v\in\dot{H}^1(\R^d,\R^N)\setminus\{0\}$, up to a subsequence. By the Rellich-Kondrachov theorem, we may also assume that $v_n$ converges to $v$ a.e. on $\R^d$. Notice that $(v_n)$ is still a maximizing sequence for $\cS$. Assumptions (1)-(2)-(3) on $v$ imply that we have a Br\'ezis-Lieb lemma for the function $G$ \cite[Theorem 2]{BreLie-83}, so that 
 $$\int G(v_n) = \int G(v) + \int G(v_n-v) + o_{n\to\ii}(1).$$
 One can then copy the proof of Proposition \ref{prop:MMM-2q} and the fact that for all $f\in\dot{H}^1(\R^d,\R^N)$ one has using scaling that 
 $$\int G(f) \le \cS \|\nabla f\|_{L^2}^{2d/(d-2)}$$
 to infer
 \begin{align*}
    \cS+o(1) = \int G(v_n) &= \int G(v) + \int G(v_n-v) + o(1) \\
    &\le\cS ( \|\nabla v\|_{L^2}^{2d/(d-2)}+\|\nabla(v_n-v)\|_{L^2}^{2d/(d-2)} ) + o(1)
 \end{align*}
 which imply in the limit $n\to\ii$ using the weak convergence of $(v_n)$ to $v$ that
 $$1\le \|\nabla v\|_{L^2}^{2d/(d-2)}+ \left(1-\|\nabla v\|_{L^2}^2\right)^{d/(d-2)}.$$
 Using that $d/(d-2)>1$ and $v\neq0$ we can conclude as in the proof of Proposition \ref{prop:MMM-2q} that $(v_n)$ converges strongly to $v$ in $\dot{H}^1$, so that $v$ is a maximizer for $\cS$.
\end{proof}

The above proof is a direct adaptation of the methods of Section \ref{sec:tools}, but in their article Br\'ezis and Lieb go actually way beyond by considering much more general assumptions on the functions $G$: 
\begin{enumerate}
 \item $G(0)=0$ and $\exists u_0\in\R^N, G(u_0)>0$;
 \item $\lim_{|u|\to0}|u|^{-p} G(u)\le0$, $\lim_{|u|\to\ii}|u|^{-p}G(u)\le0,\quad p=\frac{2d}{d-2}$,
 \item $\forall\delta>0,\ \exists C_\delta>0,\ \forall u,v\in\R^N,\ |G(u+v)-G(u)|\le\delta(|G(u)|+|u|^p)+C_\delta(|G(v)|+|v|^p+1)$.
\end{enumerate}
For such $G$, the above approach is too simplistic and one has to adapt it quite substantially. For instance, the function $G$ may change sign so that the set $\cC$ no longer coincides with $\dot{H}^1$. The assumption (1) implies that $\cS\in(0,+\ii)$. From assumption (2), one can thus still obtain the existence of a non-zero weak limit $v$ for maximizing sequences up to translations $(v_n)$ by the same method, and due to $\dot{H}^1$-boundedness one can still assume a.e. convergence as well. Fatou's lemma applied together with the assumption (2) on $G$ imply that $(G(v_n))$ is bounded in $L^1(\R^d)$ and that $G(v)\in L^1(\R^d)$, so that $v\in\cC$. The major difference in this more general setting is that the assumptions on $G$ are a priori too weak to ensure that we have a Br\'ezis-Lieb lemma for the function $G$ as before. There are two reasons for that: i) in Assumption (3), the function $|G(v)|+|v|^p+1$ on the right side is not integrable (due to the constant function) and ii) the quantity $\int|G(v_n-v)|$ may be unbounded since we only know $G(u)\le C|u|^p$ (and not $|G(u)|\le C|u|^p$). The key to bypass this issue is to prove a 'localized' Br\'ezis-Lieb lemma stating that for any compactly supported $\phi\in\dot{H}^1$ with $G(\phi)\in L^1$, one has 
$$\int G(v_n+\phi) = \int G(v_n) + \int G(v+\phi) - \int G(v) + o_{n\to\ii}(1),$$
so that when inserted into the bound $\int G(v_n+\phi) \le \cS \|\nabla (v_n+\phi)\|_{L^2}^{2d/(d-2)}$ one obtains in the limit $n\to\ii$
$$\cS + \int G(v+\phi) - \int G(v) \le \cS\left( 1 + \|\nabla (v+\phi)\|_{L^2}^2 - \|\nabla v\|_{L^2}^2\right)^{d/(d-2)}.$$
By a limiting argument, Br\'ezis and Lieb show that one can take $\phi(x)=v(\lambda x)-v(x)$ in this inequality for any $\lambda>0$, leading by rescaling to 
$$\cS + (\lambda^{-d}-1)\int G(v) \le \cS \left( 1 + (\lambda^{2-d}-1)\|\nabla v\|_{L^2}^2\right)^{d/(d-2)}.$$
Expanding this inequality close to $\lambda=1$ shows that $\int G(v)\ge \cS \|\nabla v\|_{L^2}^{2d/(d-2)}$ so that $\int G(v)= \cS \|\nabla v\|_{L^2}^{2d/(d-2)}$ and since $v\neq0$, we deduce that there exists a maximizer for $\cS$. Let us finally mention that Br\'ezis and Lieb are also able to deal with the critical case $d=2$, by even more involved methods.

\subsection{Sobolev inequality}

In \cite{Lieb-83b}, Lieb showed that his method also allowed to prove the existence of optimizers for the Sobolev embedding $\dot{H}^1(\R^d)\hookrightarrow L^{2d/(d-2)}(\R^d)$ for $d\ge3$. The idea is again to notice that one can look for optimizers in the set of radially symmetric decreasing functions, and that for such functions the inequality is equivalent to a one-dimensional Gagliardo-Nirenberg-Sobolev inequality for which Lieb proved the existence of optimizers as we mentioned above. Another way to prove this result is to notice that such optimizers are related to the optimizers of the HLS inequality in the special case $p=2$ and $\lambda=d-1$, since $(-\Delta)^{-1/2}$ is proportional to the convolution operator with $|\cdot|^{-(d-1)}$, so that the existence of optimizers for the Sobolev embedding follows from Theorem \ref{thm:HLS-pq} in this special case.  

We present here another approach not using rearrangement but still relying on Proposition \ref{prop:MMM-2q}, in the spirit of what we already presented about the GNS inequality. Again, one of the advantages of this approach is that it will provide a better description of optimizing sequences. Furthermore, it can be applied to any embedding $H^s(\R^d)\hookrightarrow L^{2d/(d-2s)}(\R^d)$, for which rearrangement techniques are not available.

Let $d\ge1$ and $s\in(0,d/2)$. Define $q:=2d/(d-2s)$ and 
$$\cS:=\sup\left\{ \int_{\R^d}|u|^q\ :\ u\in\dot{H}^s(\R^d),\ \|u\|_{\dot{H}^s}=1\right\}.$$

\begin{proposition}\label{prop:sharp-Sob}
 There exists $u\in\dot{H}^s(\R^d)$ such that $\|u\|_{\dot{H}^s}=1$ and $\|u\|_{L^q}^q=\cS$.
\end{proposition}

The proof is the same as the alternative proof of Proposition \ref{prop:sharp-GNS} using Proposition \ref{prop:MMM-2q}. Again, Assumption (5) follows from the Rellich-Kondrachov theorem so that all boils down to finding a maximizing sequence for $\cS$ which has a non-zero weak limit. While for the GNS inequality, we saw that the main enemy was the invariance by translations, here we will see that we have to deal with both translations and dilations. The main tool to break these invariances is the following result. 

\begin{proposition}[Refined Sobolev inequality \cite{GerMeyOru-97}]\label{prop:refined-Sob}
Let $\chi\in C^\ii_0(0,+\ii)$ be such that $\chi\equiv1$ in a neighborhood of $0$. Then, there exist $C>0$ and $\theta\in(0,1)$ such that for any $u\in\dot{H}^s(\R^d)$ one has
$$\|u\|_{L^{q}} \le C \Big(\sup_{t>0}t^{(d-2s)/4}\|\chi(-t\Delta)u\|_{L^\ii} \Big)^\theta\|u\|_{\dot{H}^s}^{1-\theta}.$$
\end{proposition}

This is called a refined inequality because it implies the Sobolev inequality; indeed we have for all $t>0$
 \begin{align*}
    \|e^{t\Delta}u\|_{L^\ii} &\le (2\pi)^{-d/2}\|\cF(\chi(-t\Delta)u)\|_{L^1} \\
    &= c\int_{\R^d}|\chi(t|\xi|^2)\hat{u}(\xi)|\,d\xi\\
    &\le c'\left(\int_{\R^d}\frac{|\chi(t|\xi|^2)|^2}{|\xi|^{2s}}\,d\xi\right)^{1/2}\|u\|_{\dot{H}^s}\\
    &= c't^{-(d-2s)/4}\left(\int_{\R^d}\frac{|\chi(|\xi|^2)|^2}{|\xi|^{2s}}\,d\xi\right)^{1/2}\|u\|_{\dot{H}^s}.
 \end{align*}

\begin{proof}[Proof of Proposition \ref{prop:refined-Sob}]
As stated by G\'erard in \cite{Gerard-98}, one can use the method of Chemin and Xu \cite{CheXu-97} to obtain refined inequalities. Let thus $u\in\dot{H}^s(\R^d)$. We have 
$$\|u\|_{L^q}^q=q\int_0^\ii|\{|u|>a\}|a^{q-1}\,da.$$
For any fixed $a>0$, choosing $\beta=\beta_a>0$ such that 
$$c_0\beta^{-(d-2s)/4} = a/2,\quad c_0:=\sup_{t>0}t^{(d-2s)/4}\|\chi(-t\Delta)u\|_{L^\ii},$$
we have $\|\chi(-\beta\Delta)u\|_{L^\ii}\le a/2$. We deduce 
\begin{align*}
|\{|u|>a\}| &\le |\{|\chi(-\beta\Delta)u|>a/2\}|+|\{|(1-\chi(-\beta\Delta))u|>a/2\}|\\
&=|\{|(1-\chi(-\beta\Delta))u|>a/2\}| \\
&\le \frac{4\|(1-\chi(-\beta\Delta))u\|_{L^2}^2}{a^2}.
\end{align*}
As a consequence, using the relation between $\beta_a$ and $a$,
\begin{align*}
 \|u\|_{L^q}^q &\le 4q\int_{\R^d}|\hat{u}(\xi)|^2\int_0^\ii \left(1-\chi(\beta_a|\xi|^2)\right)^2 a^{q-3}\,da\,d\xi \\
 &= 4q\int_{\R^d}|\hat{u}(\xi)|^2\int_0^\ii \left(1-\chi((a/(2c_0))^{-4/(d-2s)}|\xi|^2)\right)^2 a^{q-3}\,da\,d\xi \\
 &= 2d(2c_0)^{q-2}\int_0^\ii \left(1-\chi(b)\right)^2 b^{-s-1}\,db\int_{\R^d}|\xi|^{(q-2)(d-2s)/2}|\hat{u}(\xi)|^2\,d\xi,
\end{align*}
which proves the result since $(q-2)(d-2s)/2=2s$.
\end{proof}

Let us now explain why the refined Sobolev inequality allows to break symmetries to find a non-zero weak limit (this argument is implicit in \cite{Gerard-98}; see also \cite{KilVis-book}). Indeed, let $(u_n)$ be a maximizing sequence for $\cS$. Since $\|u_n\|_{L^{q}}^{q}\to\cS\neq0$ as $n\to\ii$, Proposition \ref{prop:refined-Sob} implies that there exists $c>0$ such that for all $n$,
$$\sup_{t>0}t^{(d-2s)/4}\|\chi(-t\Delta)u_n\|_{L^\ii}\ge c,$$
hence for all $n$ there exist $t_n>0$ and $x_n\in\R^d$ such that 
$$t_n^{(d-2s)/4}|(\chi(-t_n\Delta)u_n)(x_n)|=(2\pi)^{-d/2}t_n^{(d-2s)/4}\left|\int_{\R^d}\chi(t_n|\xi|^2)e^{-ix_n\cdot\xi}\hat{u_n}(\xi)\,d\xi\right|\ge c/2.$$
Defining $g=\cF^{-1}(\chi(|\cdot|^2))\in L^2(\R^d)$, this implies that for all $n$,
$$|\langle g,t_n^{(d-2s)/2}u_n(t_n^{1/2}(\cdot-x_n))\rangle_{L^2}|\ge (2\pi)^{d/2}c/2,$$
and hence the sequence $(t_n^{(d-2s)/2}u_n(t_n^{1/2}(\cdot-x_n)))$, which is still a maximizing sequence for $\cS$ and as such converges weakly up to a subsequence, has a weak limit $v\neq0$ (since $|\langle g,v\rangle_{L^2}|\ge(2\pi)^{d/2}c/2>0$). Notice that this approach shows that any maximizing sequence for $\cS$ converges strongly in $\dot{H}^s$ up to translations and dilations (up to a subsequence), a result that is originally due to Lions \cite[Theorem I.1]{Lions-85}. Arguments closely related to those described here were applied in the setting of the Heisenberg group in \cite[Proposition 4.3]{FraLie-12}. 

\subsection{Generalized Gagliardo-Nirenberg-Sobolev inequality}

The previous approach can also be used for the subcritical embeddings $H^s(\R^d)\hookrightarrow L^q(\R^d)$ with $q\in(2,2d/(d-2s))$. We already treated the case $s=1$ with either rearrangement methods or Proposition \ref{prop:non-zero-weak-W1p}. For general $s$, rearrangement cannot be used in the same way (since the inequality $\|f\|_{\dot{H}^s} \ge \|f^*\|_{\dot{H}^s}$ is expected to fail for general $s$) so we explain how to adapt Proposition \ref{prop:non-zero-weak-W1p} in this case. This strategy was used in \cite{BelFraVis-14}.

Let $d\ge1$, $s\in(0,d/2)$, and $2<q<2d/(d-2s)$. Define 
$$\cS:=\sup\left\{\int_{\R^d}|u|^q\ :\ u\in H^s(\R^d),\ \|u\|_{H^s}=1\right\}.$$

\begin{proposition}\label{prop:exist-GNS}
 There exists $u\in H^s(\R^d)$ with $\|u\|_{H^s}=1$ such that $\|u\|_{L^q}^q=\cS$. 
\end{proposition}

As in the previous arguments, it is enough to find a non-zero weak limit for some maximizing sequences, which we again do using a refined version of the inequality. 

\begin{proposition}\label{prop:refined-GNS}
 Define $s'\in(0,s)$ such that $q=2d/(d-2s')$. Let $\chi\in C^\ii_0(0,+\ii)$ be such that $\chi\equiv1$ in a neighborhood of $0$. Then, there exists $C>0$ such that for all $u\in H^s(\R^d)$ we have 
 \begin{equation}\label{eq:refined-GNS}
    \|u\|_{L^q}\le C\Big(\sup_{t>0}t^{(d-2s')/4}\|\chi(-t\Delta)u\|_{L^\ii}\Big)^\theta\|u\|_{H^s}^{1-\theta}. 
 \end{equation}
\end{proposition}

Proposition \ref{prop:refined-GNS} directly follows from Proposition \ref{prop:refined-Sob} and the injection $H^{s}\hookrightarrow H^{s'}$. Now let $(u_n)\subset H^s(\R^d)$ be a maximizing sequence for $\cS$. As before, from $\cS\neq0$ and the boundedness of $(u_n)$ in $H^s(\R^d)$ we deduce that there exists $c>0$ such that for all $n$,
$$\sup_{t>0}t^{(d-2s')/4}\|\chi(-t\Delta)u_n\|_{L^\ii}\ge c.$$
Now notice that from the boundedness of $(u_n)$ in $H^s$, we have for all $t>0$,
$$t^{(d-2s')/4}\|\chi(-t\Delta)u_n\|_{L^\ii}\le\frac{t^{(d-2s')/4}}{(2\pi)^{d/2}}\int_{\R^d}|\chi(t|\xi|^2)\hat{u_n}(\xi)|\,d\xi\le
\begin{cases}
 c t^{-s'/2}\|u_n\|_{L^2}\le Ct^{-s'/2}\\
 c t^{(s-s')/2}\|u_n\|_{\dot{H}^s}\le Ct^{(s-s')/2}
\end{cases}
$$
hence $t^{(d-2s')/4}\|\chi(-t\Delta)u_n\|_{L^\ii}\to0$ as $t\to0$ or $t\to+\ii$ uniformly in $n$, and thus there exist $t_-,t_+\in(0,+\ii)$ so that for all $n$,
$$\sup_{t\in[t_-,t+]}t^{(d-2s')/4}\|\chi(-t\Delta)u_n\|_{L^\ii}=\sup_{t>0}t^{(d-2s')/4}\|\chi(-t\Delta)u_n\|_{L^\ii}\ge c.$$
As above, one deduces from this lower bound that there exist $(x_n)\subset\R^d$ and $(t_n)\subset[t_-,t_+]$ such that $(t_n^{(d-2s')/4}\phi_n(t_n^{1/2}(\cdot-x_n))$ has a non-zero weak limit (up to a subsequence) in $H^s(\R^d)$. Since $(t_n)\subset[t_-,t_+]$, one can furthermore extract a subsequence so that $t_n\to t_*\in[t_-,t_+]$. This implies that $(u_n(\cdot-x_n))$ has a non-zero weak limit in $H^s(\R^d)$.

In the above proof, we used the refined inequality \eqref{eq:refined-GNS} which looks like the critical one of Proposition \ref{prop:refined-Sob} in the sense that both dilations and translations seem to appear in it (in the supremum in both $t$ and $x$). Our reasoning above shows that, by subcriticality of $q$, the supremum over all dilations $t\in(0,+\ii)$ can be replaced by a supremum over dilation $t\in[t_-,t_+]$ which is now a 'compact' symmetry group, and hence disappears in the final result. There are ways to obtain a refined inequality in which only translations appear; for instance in the case $s=1$ one has \cite[Lem. I.1]{Lions-84b}: for all $q\in(2,2d/(d-2))$ with $d\ge3$, 
$$\forall\phi\in H^1(\R^d),\quad \|\phi\|_{L^q}\le C\left(\sup_{z\in\Z^d}\|\phi\|_{L^2(z+[0,1)^d)}\right)^\theta\|\phi\|_{H^1}^{1-\theta},$$
for some $\theta\in(0,1)$ and $C>0$ independent of $\phi$. From this inequality, it is not hard to deduce the existence of non-zero weak limits up to translations. For general $s$, the non-locality of the $H^s$-norm makes the proof more difficult but a similar inequality where only translation appears was proved for $s\in(0,1)$ in \cite[Eq. (B.12)]{LenLew-11}. 

\begin{remark}
Sometimes, existence of maximizers is stated when $\cS$ is replaced by 
$$\sup\{\|u\|_{L^q}^q\ :\ u\in H^s(\R^d),\ \|u\|_{L^2}=1,\ \|u\|_{\dot{H}^s}=1\}.$$
In this case, one can still apply the same strategy because the above supremum is related by scaling to $\cS$ where the $H^s$-norm $\|u\|_{H^s}=\|(1-\Delta)^{s/2}u\|_{L^2}$ is replaced by the equivalent norm $\|(1+(-\Delta)^{s})^{1/2}u\|_{L^2}$. 
\end{remark}

\subsection{Strichartz inequality}

We give one final example where the above techniques can be applied, which is the one of Strichartz inequalities. We will see that in this case, more advanced techniques are required to obtain Assumptions (4) and (5) of Proposition \ref{prop:MMM-2q}. The fact that Lieb's strategy could be applied to this case was understood by R. Frank and the author in an unpublished work.

Let $d\ge1$ and define 
$$\cS=\sup\left\{\int_{\R}\int_{\R^d}|(e^{it\Delta_x}u)(x)|^{2+4/d}\,dx\,dt\ :\ u\in L^2(\R^d),\ \|u\|_{L^2}=1\right\}.$$

This problem has a lot of symmetries: besides translations and dilations, there are also the transformations $u(x)\to (e^{is\Delta_x}u)(x)$ for any $s\in\R$ and $u(x)\to e^{ix\cdot\xi}u(x)$ for any $\xi\in\R^d$. 

\begin{proposition}
 There exists $u\in L^2(\R^d)$ with $\|u\|_{L^2}=1$ and $\|e^{it\Delta}u\|_{L^{2+4/d}(\R\times\R^d)}^{2+4/d}=\cS$. 
\end{proposition}

This result has first been proved by Kunze \cite{Kunze-03} for $d=1$ and by Shao \cite{Shao-09} for $d\ge2$, using profile decompositions. We show here that Proposition \ref{prop:MMM-2q} leads to the same result, even if the same tools are at the core of all the proofs. To apply Proposition \ref{prop:MMM-2q} to this case, we again need to find a maximizing sequence $(v_n)$ which converges weakly to a non-zero limit $v$, as well as to show that $e^{it\Delta_x}v_n\to e^{it\Delta_x}v$ a.e. on $\R\times\R^d$. 

The first step can also be obtained via the following refined inequality, which can be found for instance in the lecture notes of Killip and Visan \cite[Prop. 4.24]{KilVis-book} and which relies on deep bilinear estimates due to Tao \cite{Tao-03}:

\begin{proposition}
 There exists $C>0$ and $\theta\in(0,1)$ such that for all $u\in L^2(\R^d)$ we have 
  $$\|e^{it\Delta_x}u\|_{L^{2+4/d}(\R\times\R^d)}\le C\left(\sup_{Q\in\cD}|Q|^{-1/2}\|e^{it\Delta_x}u_Q\|_{L^\ii(\R\times\R^d)}\right)^{\theta}\|u\|_{L^2}^{1-\theta},$$
  where $\cD$ denotes the family of dyadic cubes on $\R^d$ of side length $2^j$ and centered at $(2^j\Z)^d$, for all $j\in\Z$, and $u_Q:=\cF^{-1}(\1_Q\hat{u})$.
\end{proposition}

If $(u_n)$ is a maximizing sequence for $\cS$, one deduces from the refined inequality that there exist $c>0$, $(t_n,x_n)\subset\R\times\R^d$, $(\delta_n)\subset(0,+\ii)$, and $(c_n)\subset\R^d$ such that for any $n$,
$$|\langle g,v_n\rangle|=\delta_n^{-d/2}\left|\int_{c_n+[-\delta_n,\delta_n)^d}e^{-it_n|\xi|^2-ix_n\cdot\xi}\hat{u_n}(\xi)\,d\xi\right|\ge c,$$
where $g=\cF^{-1}(\1_{[-1,1)^d})$ and $v_n:=\cF^{-1}(\xi\mapsto \delta_n^{d/2}e^{-it_n|c_n+\delta_n\xi|^2-ix_n\cdot(c_n+\delta_n\xi)}\hat{u_n}(c_n+\delta_n\xi))$. It can be shown that $(v_n)$ is still a maximizing sequence for $\cS$, which thus has a non-zero weak limit in $L^2$.

The a.e. convergence of $e^{it\Delta_x}v_n$ to $e^{it\Delta_x}v$ can be proved using the following result stated for instance in \cite[Proposition 1.4]{Keraani-01} which is a consequence of the local smoothing properties of $e^{it\Delta_x}$:

\begin{proposition}
 The map $v\in L^2(\R^d)\mapsto e^{it\Delta}v \in L^2_{\rm loc}(\R\times\R^d)$ is compact. 
\end{proposition}

\begin{proof}
 Notice first that we have the following local smoothing estimate \cite{ConSau-88}: for any $a\in\cS(\R^d)$, there exists $C>0$ such that for any $u\in L^2(\R^d)$ one has 
 $$\int_{\R}\int_{\R^d}a(x)|(-\Delta_x)^{1/4}(e^{it\Delta_x}u)(x)|^2\,dx\,dt \le C\|u\|_{L^2}^2.$$
 Indeed, the left integral in Fourier variables is equal to
 $$(2\pi)^{d/2}\int_{\R^d}\int_{\R^d}\hat{u}(\xi) K_a(\xi,\xi')\bar{\hat{u}(\xi')}\,d\xi\,d\xi'$$
 with an integral kernel
 $$K(\xi,\xi')=|\xi|^{1/2}|\xi'|^{1/2}\hat{a}(\xi'-\xi)\delta(|\xi|^2-|\xi'|^2).$$
 To prove the inequality, it is enough by Schur's test to bound uniformly in $\xi$,
 $$\int_{\R^d}|K(\xi,\xi')|\,d\xi'=|\xi|^{d-1}\int_{\Sph^{d-1}}|\hat{a}(|\xi|\omega-\xi)|\,d\omega.$$
 For small $\xi$, this is clearly bounded while for large $\xi$, one can write any $\omega\in\Sph^{d-1}$ as $\omega=(\cos\theta)\omega_\xi+(\sin\theta)\omega'$ with $\omega_\xi=\xi/|\xi|$ and $\omega'\cdot\xi=0$, $|\omega'|=1$. Using the decay of $\hat{a}$, we then have the bound 
 $$\int_{\Sph^{d-1}}|\hat{a}(|\xi|\omega-\xi)|\,d\omega\le C\int_0^\pi\frac{(\sin\theta)^{d-2}}{(1+|\xi|\sin\theta)^d}\,d\theta\le C'|\xi|^{-(d-1)}.$$
 To prove the proposition, let $(v_n)\subset L^2(\R^d)$ such that $v_n\rightharpoonup0$ and let $\cC\subset\R^{d+1}$ be a compact set. Let us prove that $e^{it\Delta}v_n\to0$ in $L^2(\cC)$. Let $\epsilon>0$, $\Lambda>0$ and $a\in\cS(\R^d)$ such that $a>0$. Split $v_n$ as $v_n=v_{n,>}+v_{n,<}$ with $\hat{v_{n,>}}(\xi)=\1(|\xi|>\Lambda)\hat{v_n}$. By the local smoothing estimate, we have 
 \begin{align*}
  \|e^{it\Delta}v_{n,>}\|_{L^2(\cC)}^2 &\le (\min_\cC a)^{-1}\int_{\R}\int_{\R^d}a(x)|(e^{it\Delta}v_{n,>})(x)|^2\,dx\,dt \\
  &\le C(\min_\cC a)^{-1}\|(-\Delta)^{-1/4}v_{n,>}\|_{L^2}^2 \\
  &\le C'\Lambda^{-1},
 \end{align*}
 where in the last step we used that $(v_n)$ is bounded in $L^2(\R^d)$. We thus fix $\Lambda$ large enough so that $C'\Lambda^{-1}\le\epsilon^2$, hence $\|e^{it\Delta}v_{n,>}\|_{L^2(\cC)}\le\epsilon$ for all $n$. Now for any $(t,x)\in\cC$ we have 
 $$(e^{it\Delta}v_{n,<})(x)=(2\pi)^{-d/2}\int_{|\xi|\le\Lambda}e^{-it|\xi|^2+i\xi\cdot x}\hat{v_n}(\xi)\,d\xi=\langle v_n,g\rangle$$
 with $g:=(2\pi)^{-d/2}\cF^{-1}(\1(|\xi|\le\Lambda)e^{it|\xi|^2-i\xi\cdot x})\in L^2(\R^d)$, so that by weak convergence we have $(e^{it\Delta}v_{n,<})(x)\to0$ as $n\to\ii$ for all $(t,x)\in\cC$. Furthermore, again by boundedness of $(v_n)$ in $L^2$, we also have the bound $|(e^{it\Delta}v_{n,<})(x)|\le C\Lambda^{d/2}$ uniform in $n$. By dominated convergence, we deduce that $e^{it\Delta}v_{n,<}\to0$ in $L^2(\cC)$ and hence $\|e^{it\Delta}v_{n,<}\|_{L^2(\cC)}\le\epsilon$ for $n$ large enough, which concludes the proof.
\end{proof}

\section{Applications to problems with approximate symmetries}\label{sec:app-approx-symp}

After these several examples of how to deal with non-compact invariances, we present a few examples where the same strategy can be applied in the context of 'almost' invariances. What we mean by that is for instance a problem that is not translation-invariant but when one translates a function to infinity, a new effective optimization problem arises. To show that the original problem has an optimizer through convergent optimizing sequences, one thus has to understand why it is energetically unfavorable to send some/all the mass to infinity. We illustrate this idea on the historical example of the Br\'ezis-Nirenberg problem, which is a version of the Sobolev inequality that is not translation or dilation invariant, and on the more recent example of the Stein-Tomas inequality, which is a version of the Strichartz inequality with less invariances. We focus on these two examples but of course, this phenomenon is ubiquitous in optimization problems and we will not try to give a complete list of the various methods used to deal with it.

\subsection{The Br\'ezis-Nirenberg problem} Let $d\ge3$ and $\Omega$ an open bounded subset of $\R^d$. Denote by $\lambda_1(\Omega)>0$ the first eigenvalue of the Dirichlet Laplacian on $\Omega$. Let $\lambda\in[0,\lambda_1(\Omega))$ and $q\in2d/(d-2)$. In the Br\'ezis-Nirenberg problem \cite{BreNir-83}, one wants to know whether maximizers exist for
$$\cS_\lambda:=\sup\left\{ \int_\Omega |u|^q\ :\ u\in H^1_0(\Omega),\ \|\nabla u\|_{L^2}^2-\lambda\|u\|_{L^2}^2=1\right\}.$$
Compared to the similar problem of the existence of optimizers for the Sobolev inequality in $\R^d$, we see that posing the problem in $\Omega$ breaks translation and dilation invariance. However, dilation remains an 'almost' invariance, in the sense that if one fixes $x_0\in\Omega$ and $\epsilon>0$ such that $B(x_0,\epsilon)\subset\Omega$, then for any $u\in C^\ii_c(B(x_0,\epsilon))\subset H^1_0(\Omega)$, the function $u_\delta:x\mapsto \delta^{(d-2)/2}u(x_0+\delta(x-x_0))$ belongs to $H^1_0(\Omega)$ for any $\delta\ge1$ so that one can still dilate some functions in the maximization set (only for large dilation parameters $\delta$, meaning that one can only contract a function). The operation $u\to u_\delta$ clearly does not leave the maximization problem invariant, since while $\int_\Omega|u_\delta|^q = \int_\Omega|u|^q$ and $\|\nabla u_\delta\|_{L^2}^2=\|\nabla u\|_{L^2}^2$, we have $\|u_\delta\|_{L^2}^2=\delta^{-2}\|u\|_{L^2}$. This fact has several important consequences regarding the maximization problem $\cS_\lambda$. First, it implies that $\cS_\lambda\ge\cS$, where $\cS$ is the Sobolev constant on $\R^d$,
$$\cS=\sup\left\{\int_{\R^d} |u|^q\ :\ u\in H^1(\R^d),\ \|\nabla u\|_{L^2}=1\right\}.$$
Indeed, by density of $C^\ii_0(\R^d)$ in $H^1(\R^d)$, the above sup can be computed by looking only at $u$'s in  $C^\ii_0(\R^d)$. Now any such $u$ can be dilated and translated so that it belongs to $C^\ii_c(B(x_0,\epsilon))\subset H^1_0(\Omega)$, without changing its $L^q(\R^d)$ and $\dot{H}^1(\R^d)$ norms. By the above procedure, one can then dilate this $u$ so that
$$\frac{\int_\Omega|u_\delta|^q}{(\|\nabla u_\delta\|_{L^2}^2-\lambda\|u_\delta\|_{L^2}^2)^{q/2}}\longrightarrow_{\delta\to\ii}\frac{\int_{\R^d}|u|^q}{\|\nabla u\|_{L^2}^{q/2}},$$
which then indeed proves that $\cS\le\cS_\lambda$ since the left quotient is less than $\cS_\lambda$ for all $\delta\ge1$. Secondly, this 'almost' invariance also has important consequences regarding the convergence of maximizing sequences. Indeed, if one has the \emph{equality} $\cS=\cS_\lambda$, then one can easily find maximizing sequences for $\cS_\lambda$ that converge weakly to $0$ in $H^1_0(\Omega)$: one can just consider a maximizing sequence $(u_n)\subset C^\ii_0(\R^d)$ for $\cS$ and scale it by the above procedure, with a scaling parameter $\delta_n\ge1$ large enough so that $\delta_n^{-1}\|u_n\|_{L^2}\to0$ as $n\to\ii$. Hence, $\cS_\lambda<\cS$ is a necessary condition for maximizing sequences for $\cS_\lambda$ to converge strongly. Using the same method as in Proposition \ref{prop:MMM-2q}, Lieb proved in \cite[Lemma 1.2]{BreNir-83} that this condition was also sufficient.

\begin{proposition}
 If $\cS_\lambda>\cS$, then there exists $u\in H^1_0(\Omega)$ with $\|\nabla u\|_{L^2}^2-\lambda\|u\|_{L^2}^2=1$ such that $\|u\|_{L^q}^q=\cS_\lambda$.
\end{proposition}

\begin{proof}
 Define 
 $$\tilde{\cS}:=\sup\left\{\limsup_{n\to\ii}\int_\Omega |u_n|^q\ :\ (u_n)\subset H^1_0(\Omega),\ \forall n,\ \|\nabla u_n\|_{L^2}^2-\lambda\|u_n\|_{L^2}^2=1,\ u_n\rightharpoonup0\right\},$$
 and let us first show the result if $\cS_\lambda>\tilde{\cS}$. We will then show that $\cS=\tilde{\cS}$. Let $(u_n)$ be a maximizing sequence for $\cS_\lambda$. Since $\cS_\lambda>\tilde{\cS}$, we deduce that $u_n\not\rightharpoonup0$ and hence $u_n\rightharpoonup u\neq0$ in $H^1_0(\Omega)$, up to a subsequence. By the Rellich-Kondrachov theorem, we may also assume that $u_n\to u$ a.e. on $\Omega$. Then, one can apply Proposition \ref{prop:MMM-2q} to the sequence $(u_n)$ and to the operator $A:u\in H^1_0(\Omega)\mapsto u\in L^q(\Omega)$, where the Hilbert space $H^1_0(\Omega)$ is endowed with the Hilbert space norm $\|v\|=(\|\nabla u\|_{L^2}^2-\lambda\|u\|_{L^2}^2)^{1/2}$. We obtain in this way that $u$ is a maximizer for $\cS_\lambda$ and that $u_n\to u$ strongly in $H^1_0(\Omega)$. It thus remains to prove that $\cS=\tilde{\cS}$. First, our construction above of a maximizing sequence for $\cS$ concentrating at a point in $\Omega$ shows that $\cS\le\tilde{\cS}$. Let us show the reverse inequality. To do so, let $(u_n)\subset H^1_0(\Omega)$ be such that $\|\nabla u_n\|_{L^2}^2-\lambda\|u_n\|_{L^2}^2=1$ for all $n$ and such that $u_n\rightharpoonup0$ in $H^1_0(\Omega)$. Again by the Rellich-Kondrachov theorem, we deduce that $u_n\to0$ strongly in $L^2(\Omega)$, so that $\|\nabla u_n\|_{L^2}\to1$ as $n\to\ii$. Defining $v_n\in H^1(\R^d)$ to be the extension of $u_n$ by $0$ outside of $\Omega$, we thus have for all $n$
 $$\frac{\int_\Omega|u_n|^q}{1+o(1)}=\frac{\int_{\R^d}|v_n|^q}{\|\nabla v_n\|_{L^2}^{q/2}}\le\cS,$$
 hence $\limsup_{n\to\ii}\int_\Omega|u_n|^q\le\cS$, so that $\tilde{\cS}\le\cS$. 
\end{proof}

\begin{remark}
 In \cite{BreNir-83}, Br\'ezis and Nirenberg showed that $\cS_\lambda>\cS$ is true for all $\lambda\in(0,\lambda_1(\Omega))$ if $d\ge4$. On the other hand, if $d=3$, they show that for any bounded smooth $\Omega$ (in fact, finite measure is sufficient) there is a $\lambda_*(\Omega)\in(0,\lambda_1(\Omega))$ such that $\cS_\lambda=\cS$ for all $\lambda\le\lambda_*(\Omega)$ and $\cS_\lambda>\cS$ for $\lambda>\lambda_*(\Omega)$.
\end{remark}

\subsection{The Stein-Tomas inequality}

In \cite{FraLieSab-16}, Lieb and his coauthors applied similar ideas to study the existence of maximizers for the Stein-Tomas inequality \cite{Stein-86,Tomas-75},
$$\cS=\sup\left\{\int_{\R^{d+1}}|\check{f}|^q\ :\ f\in L^2(\Sph^d),\ \|f\|_{L^2}=1 \right\}$$
with $d\ge1$ and $q=2+4/d$, and where for any $f\in L^2(\Sph^d)$ we defined 
$$\forall x\in\R^{d+1},\ \check{f}(x)=\int_{\Sph^d}f(\omega)e^{-ix\cdot\omega}\,d\omega.$$
Since the problem is posed on the sphere, the problem is rotation invariant but since the group of rotations is compact, it does not induce any loss of compactness. However, similarly to the Br\'ezis-Nirenberg problem, a potential loss of compactness may arise from functions that concentrate at a point on the sphere. Interestingly, the effective problem that one finds in this case is the Strichartz inequality that we mentioned above, for which we denote the sharp constant by $\cS_{\rm Stri}$. There is a twist compared to the Br\'ezis-Nirenberg problem, though: one can show that it is energetically more favorable to concentrate around \emph{two} points (namely, antipodal points on the sphere) rather than at a single point (with the mass equally split between the two points). One can compute explicitly that such a scenario leads to an effective problem which sharp constant is $a\cS_{\rm Stri}$ for some $a>1$. In \cite{FraLieSab-16}, the authors prove that this scenario is the only possible source of loss of compactness:

\begin{proposition}
 If $\cS>a\cS_{\rm Stri}$, then there exists $f\in L^2(\Sph^d)$ with $\|f\|_{L^2}=1$ such that $\int_{\R^{d+1}}|\check{f}|^q=\cS$.
\end{proposition}

The proof of this result follows the same line of the one above for the Br\'ezis-Nirenberg problem: defining
$$\tilde{\cS}=\sup\left\{\limsup_{n\to\ii}\int_{\R^{d+1}}|\check{f_n}|^q\ :\ (f_n)\subset L^2(\Sph^d),\ \forall n,\, \|f_n\|_{L^2}=1,\ f_n\rightharpoonup_{\rm sym}0\right\},$$
one can show using Proposition \ref{prop:MMM-2q} that if $\cS>\tilde{\cS}$, then maximizing sequences converge strongly to a maximizer for $\cS$ since by definition of $\tilde{\cS}$, they admit non-zero weak limits. Here, the notation $f_n\rightharpoonup_{\rm sym}0$ refers to weak convergence to zero ``up to symmetry'', because this problem still has a non-compact invariance (the transformation $f(\omega)\to e^{-i x\cdot\omega}f(\omega)$ for any $x\in\R^{d+1}$). Hence, $f_n\rightharpoonup_{\rm sym}0$ means that for all $(x_n)\subset\R^{d+1}$, $e^{-ix_n\cdot\omega}f_n\rightharpoonup0$ in $L^2(\Sph^d)$. If $\cS>\tilde{\cS}$, we thus deduce that maximizing sequences $(f_n)$ for $\cS$ are such that there exists $(x_n)\subset\R^{d+1}$ such that $(e^{-ix_n\cdot\omega}f_n)$ has a non-zero weak limit in $L^2(\Sph^d)$ (and it is still a maximizing sequence to which we can apply Proposition \ref{prop:MMM-2q}). In the end, everything boils down to proving that $\tilde{\cS}=a\cS_{\rm Stri}$. Notice again that due to our explicit examples of functions that concentrate at two antipodal points, we always have $\tilde{\cS}\ge a\cS_{\rm Stri}$ and one has to prove the reverse inequality. It follows from a refined version of the Stein-Tomas inequality (in the spirit of the refined inequalities that we mentioned above) that for a sequence $(f_n)$ which is an almost maximizer for $\tilde{\cS}$ (so that $\check{f_n}$ does not converge strongly to $0$ in $L^q(\R^{d+1})$), one can find a sequence of scaling parameters such that some positive mass of $(f_n)$ concentrates around antipodal points at this scale. Using a version of Proposition \ref{prop:MMM-2q} adapted to this setting, one can show that actually all the mass of $(f_n)$ concentrates around these points (for otherwise it would violate its almost maximality), so that it becomes a competitor for the maximization problem $a\cS_{\rm Stri}$, leading to $\tilde{\cS}\le a\cS_{\rm Stri}$.

\begin{remark}
 It is known that $\cS>a\cS_{\rm Stri}$ holds for $d=1$ \cite{Shao-15} and $d=2$ \cite{ChrSha-12a}, but it is a conjecture for $d\ge3$. In \cite{FraLieSab-16}, it is shown that this conjecture would be a consequence of the conjectured fact that Gaussians are maximizers for the Strichartz problem $\cS_{\rm Stri}$. Let us also mention that a similar strategy was applied to the case where the sphere is replaced by the one-dimensional cubic curve $y=x^3$ in \cite{FraSab-18}.
\end{remark}

Finally, let us mention that the fact that a strict inequality between the supremum and an 'effective' supremum obtained using the almost invariances of the problem implies the existence of optimizers appeared in many places in the literature; and perhaps one of the earliest example of such a phenomenon is the HVZ theorem in many-body quantum mechanics going back to the 1960s\footnote{We thank Mathieu Lewin for pointing this out to us.} (see for instance \cite[Theorem 3.1]{Lewin-11} for a version reflecting these compactness ideas; Zhislin's original proof \cite{Zhislin-60} being very close in spirit, as manifested by \cite[Eq. (2.14)]{Zhislin-69} which shows that sequences of trial functions vanishing weakly must have at least the energy of $N-1$ particles), where it is proved that an atom can bind $N$ electrons if and only if the $N$-body ground state is strictly less that the $(N-1)$-body ground state (in other words, it is energetically unfavorable to send one of the electrons to infinity). In physics, it is also standard that this kind of inequalities imply the stability of the system; as demonstrated for instance in the work of Bethe \cite{Bethe-29}, where the fact that adding an electron to the hydrogen atom is proven to be energetically favorable\footnote{We thank Rupert Frank for mentioning this reference as well as Zhislin's one.}. There, the fact that it implies the stability of the negative hydrogen ion (that is, the existence of a ground state) is implicit and not even mentioned. A famous problem where this phenomenon also appears is the Yamabe problem in Riemannian geometry, where optimizers exist as soon as concentration around a point is energetically unfavorable, as understood by Aubin \cite{Aubin-76}. Finally, this is also the content of the \emph{strict binding inequalities} introduced by Lions in his concentration-compactness theory, which are shown to be necessary and sufficient conditions for the precompactness of optimizing sequences and which were used by Lions and others to study various types of optimization problems.

\appendix

\section{Profile decompositions}\label{app:profile}

We explain why the techniques presented in this review naturally lead to a structure theorem for bounded sequences in some function spaces which is called \emph{profile decomposition}, useful in many analysis problems. For general sequences, this was first introduced by G\'erard \cite{Gerard-98} but they appear in many works that we will not try to list here. We state it in the context of $H^s(\R^d)$ (which was proved before for instance in \cite[Lem. 11]{LenLew-11}) but as the proof will show, its basic ingredients are the same as the ones we used to prove the existence of optimizers: namely, a refined inequality  to detect at which scale some positive mass lives, a compactness tool ensuring a.e. convergence, and the Br\'ezis-Lieb lemma. Hence, a similar profile decomposition can be inferred for problems where such ingredients are present. That these tools could be used to obtain a profile decomposition was understood quite early by Nawa \cite{Nawa-90,Nawa-94} in the context of the nonlinear Schr\"odinger equation.

\begin{proposition}\label{prop:profile}
 Let $d\ge1$, $s\in(0,d/2)$. Let $(u_n)$ be a bounded sequence in $H^s(\R^d)$. Then, there exists a subsequence of $(u_n)$ (that we still denote by $(u_n)$), there exist $(v^j)_{j\ge1}\subset H^s(\R^d)$ and $(x^j_n)_{j\ge1}\subset\R^d$ such that defining for all $J\ge1$ and $n\in\N$, $$r_n^J:=u_n-\sum_{j=1}^J v^j(\cdot+x^j_n),$$
 we have
 \begin{enumerate}
  \item $\forall J\ge1$, 
  $$\|u_n\|_{H^s}^2=\sum_{j=1}^J\|v^j\|_{H^s}^2 + \|r_n^J\|_{H^s}^2+o_{n\to\ii}(1).$$
  \item $\forall q\in(2,2d/(d-2s))$, $\forall J\ge1$, 
  $$\|u_n\|_{L^q}^q=\sum_{j=1}^J\|v^j\|_{L^q}^q + \|r_n^J\|_{L^q}^q+o_{n\to\ii}(1).$$
  \item $\forall j\neq j'$, $\lim_{n\to\ii}|x_n^j-x_n^{j'}|=+\ii$;
  \item $\forall q\in(2,2d/(d-2s))$, $\lim_{J\to\ii}\limsup_{n\to\ii}\|r^J_n\|_{L^q}=0$.
 \end{enumerate}
\end{proposition}

Before proving this proposition, let us recall a consequence of the refined inequality of Proposition \ref{prop:refined-GNS} that we used implicitly in our proof of Proposition \ref{prop:exist-GNS}.

\begin{proposition}\label{prop:dicho-GNS}
 Let $d\ge1$, $s\in(0,d/2)$ and $q\in(2,2d/(d-2s))$. Then, there exist $c>0$ and $\theta\in(0,1)$ such that for any bounded sequence  $(u_n)\subset H^s(\R^d)$ we have 
 \begin{enumerate}
  \item either $\lim_{n\to\ii}\|u_n\|_{L^q}=0$;
  \item or $a:=\limsup_{n\to\ii}\|u_n\|_{L^q}>0$ and there exist a subsequence of $(u_n)$ (that we still denote by $(u_n)$), there exist $v\in H^s(\R^d)\setminus\{0\}$ and $(x_n)\subset\R^d$ such that $u_n(\cdot-x_n)\rightharpoonup v$ in $H^s(\R^d)$ and such that $\|v\|_{H^s}\ge c a^{1/\theta}b^{1-1/\theta}$, where $b:=\liminf_{n\to\ii}\|u_n\|_{H^s}$.
 \end{enumerate}
\end{proposition}

What we did not emphasize in the proof of Proposition \ref{prop:exist-GNS} was the lower bound on the norm of $v$; but recall that, using the refined inequality, the proof provides a $g\in H^s$ and a $c'>0$ such that $|\langle g,v\rangle|\ge c'a^{1/\theta}b^{1-1/\theta}$, so that the lower bound on $\|v\|_{H^s}$ just follows from the Cauchy-Schwarz inequality.

\begin{proof}[Proof of Proposition \ref{prop:profile}]
 Define $b:=\liminf_{n\to\ii}\|u_n\|_{H^s}$ and let us fix any $q\in(2,2d/(d-2s)$. If $\lim_{n\to\ii}\|u_n\|_{L^q}=0$, then we define $v^j=0$ for all $j$ and pick any $(x^j_n)\subset\R^d$ satisfying (3) to get the result. If $a_0:=\limsup_{n\to\ii}\|u_n\|_{L^q}>0$, we apply Proposition \ref{prop:dicho-GNS} to infer that there exist $v^1\in H^s(\R^d)\setminus\{0\}$ with $\|v^1\|_{H^s}\ge ca_0^{1/\theta}b^{1-1/\theta}$ and $(x^1_n)\subset\R^d$ such that $u_n(\cdot-x^1_n)\rightharpoonup v^1$ in $H^s(\R^d)$. Recall that we then define $r^1_n:=u_n-v^1(\cdot+x^1_n)$ so that $r^1_n(\cdot-x^1_n)\rightharpoonup0$ and thus we have 
 $$\|u_n\|_{H^s}^2=\|u_n(\cdot-x^1_n)\|_{H^s}^2=\|v^1+r^1_n(\cdot-x^1_n)\|_{H^s}^2=\|v^1\|_{H^s}^2+\|r^1_n\|_{H^s}^2+o_{n\to\ii}(1),$$
 which is (1) for $J=1$. To obtain (2) for $J=1$, one can use that $r^1_n(\cdot-x^1_n)\rightharpoonup0$ together with the boundedness of $(r^1_n)$ in $H^s$ and the Rellich-Kondrachov theorem to infer that, up to a subsequence, $r^1_n(\cdot-x^1_n)\to0$ a.e. on $\R^d$. Then, by the Br\'ezis-Lieb lemma we obtain for any $r\in(2,2d/(d-2s)$,
 $$\|u_n\|_{L^r}^r=\|u_n(\cdot-x^1_n)\|_{L^r}^r=\|v^1+r^1_n(\cdot-x^1_n)\|_{L^r}^r=\|v^1\|_{L^r}^r+\|r^1_n\|_{L^r}^r+o_{n\to\ii}(1),$$
 which is (2) for $J=1$. If $\lim_{n\to\ii}\|r^1_n\|_{L^q}=0$, then we define $v^j=0$ for all $j\ge2$ and pick any $(x^j_n)_{j\ge2}$ so that (3) is satisfied, and then we get the result. If $a_1:=\limsup_{n\to\ii}\|r^1_n\|_{L^q}>0$ we continue by applying the previous procedure to $(r^1_n)$ instead of $(u_n)$. By the property (1) for $J=1$, we have $\liminf_{n\to\ii}\|r^1_n\|_{H^s}\le\liminf_{n\to\ii}\|u_n\|_{H^s}=b$ so that by Proposition \ref{prop:dicho-GNS}, there exist $v^2\in H^s(\R^d)\setminus\{0\}$ with $\|v^2\|_{H^s}\ge ca_1^{1/\theta}b^{1-1/\theta}$ and $(x^2_n)\subset\R^d$ such that $r_n^1(\cdot-x^2_n)\rightharpoonup v^2$ in $H^s(\R^d)$. With $r^2_n=r^1_n-v^2(\cdot+x^2_n)$, we thus have as previously that 
 $$\|r^1_n\|_{H^s}^2=\|v^2\|_{H^s}^2+\|r^2_n\|_{H^s}^2+o_{n\to\ii}(1),$$
 $$\|r^1_n\|_{L^r}^r=\|v^2\|_{L^r}^r+\|r^2_n\|_{L^r}^r+o_{n\to\ii}(1),\ r\in(2,2d/(d-2s))$$
 which when inserted in (1), (2) for $J=1$ give (1), (2) for $J=2$. Let us now check that $|x^1_n-x^2_n|\to\ii$. Assume by contradiction that some subsequence of $(x^1_n-x^2_n)$ is bounded; then up to subsequence we may assume that $y_n:=x^1_n-x^2_n\to y_\ii\in\R^d$. Recall that, by construction, $r^1_n(\cdot-x^2_n)\rightharpoonup v^2$ so that $r^1_n(\cdot-x^1_n)=r^1_n(\cdot-x^2_n-y_n)\rightharpoonup v^2(\cdot-y_\ii)$. This is a contradiction with $v^2\neq0$ and $r^1_n(\cdot-x^1_n)\rightharpoonup0$. We thus obtain (3) for $j=1$ and $j'=2$. We can iterate this construction by again distinguishing between $\lim_{n\to\ii}\|r^2_n\|_{L^q}=0$ and $a_2:=\limsup_{n\to\ii}\|r^2_n\|_{L^q}>0$ to build all the profiles $(v^j)$ and the centers $(x^j_n)$. It remains to check (4), which is non-trivial only in the case where all the profiles $(v^j)$ are non-zero (that is, $a_J>0$ for all $J$). Assume by contradiction that (4) fails in this case, that is there exist $\epsilon>0$ and a sequence $(J_k)$ such that $a_{J_k}\ge\epsilon$ for all $k$. We deduce that for all $k$, $\|v_{J_k}\|_{H^s}\ge c\epsilon^{1/\theta}b^{1-1/\theta}>0$ (where we recall that $c,\epsilon,b$ are independent of $k$), a contradiction with $\sum_{j=1}^\ii \|v^j\|_{H^s}^2\le b^2$ which is a consequence of (1).  
\end{proof}


\end{document}